\documentclass[a4paper,11pt]{amsart}
\pdfoutput=1

\usepackage[utf8x]{inputenc}
\usepackage[english]{babel}
\usepackage{geometry}
\usepackage{fullpage}
\usepackage{hyperref}
\usepackage{amsmath}
\usepackage{amsthm}
\usepackage{amsfonts}
\usepackage{amssymb}
\usepackage{graphicx}
\usepackage{epsfig}
\usepackage{enumerate}
\usepackage{mathrsfs}
\usepackage{xfrac}
\usepackage{tikz}
\usetikzlibrary{arrows}
\usepackage{latexsym}
\usepackage{indentfirst}

\swapnumbers
\theoremstyle{plain}
\newtheorem{prop}{Proposition}[section]
\newtheorem{teorema}[prop]{Theorem}
\newtheorem{lemma}[prop]{Lemma}
\newtheorem{cor}[prop]{Corollary}
\theoremstyle{definition}
\newtheorem{definiz}[prop]{Definition}
\theoremstyle{remark}
\newtheorem{oss}[prop]{\textsc{Remark}}

\renewcommand{\paragraph}[1]{\noindent\addtocounter{prop}{1}\arabic{section}.\arabic{prop}. #1}

\newcommand{\bbA}{\mathbb{A}}
\newcommand{\bbC}{\mathbb{C}}
\newcommand{\bbG}{\mathbb{G}}

\newcommand{\bbP}{\mathbb{P}}

\newcommand{\bbZ}{\mathbb{Z}}

\newcommand{\barx}{\overline{x}}
\newcommand{\Hom}{\text{Hom}}
\newcommand{\h}{H_{\text{\'et}}}
\newcommand{\id}{\text{id}}
\newcommand{\insieme}[2]{\left\{\,#1\,\middle|\, #2\,\right\}}
\newcommand{\Mor}{\text{Mor}}
\newcommand{\s}{\text{s}}
\newcommand{\scrO}{\mathscr{O}}
\newcommand{\scrG}{\mathscr{G}}
\newcommand{\scrZ}{\mathscr{Z}}
\newcommand{\sh}{\text{sh}}

\renewcommand{\phi}{\varphi}

\DeclareMathOperator{\B}{B\!}
\DeclareMathOperator{\Br}{Br}
\DeclareMathOperator{\Char}{char}
\DeclareMathOperator{\coker}{coker}
\DeclareMathOperator{\gal}{Gal}

\DeclareMathOperator{\im}{im}
\DeclareMathOperator{\N}{N}
\DeclareMathOperator{\pic}{Pic}
\DeclareMathOperator{\res}{res}
\DeclareMathOperator{\sez}{\Gamma}
\DeclareMathOperator{\spec}{Spec}

\DeclareMathOperator{\val}{v}

\author[Poma]{Flavia Poma}
\title[\'Etale cohomology of a DM curve with coefficients in $\mathbb{G}_m$]{\'Etale cohomology of a DM curve-stack\\ with coefficients in $\mathbb{G}_m$}
\date{\today\\
MSC classes: 14F20 (Primary), 14A20 (Secondary)}
\address{SISSA\\
via Bonomea, 265\\
34136 Trieste\\
Italy}
\email{poma@sissa.it\\
Phone: 0403787324}

\begin{document}
\begin{abstract}
  We compute \'etale cohomology groups $\h^i(X, \mathbb{G}_m)$ in
  several cases, where $X$ is a smooth tame Deligne-Mumford stack of
  dimension $1$ over an algebraically closed field. We have complete
  results for orbicurves (and, more generally, for twisted nodal
  curves) and in the case all stabilizers are cyclic; we give partial
  results and examples in the general case. In particular, we show
  that if the stabilizers are abelian then $\h^2(X, \bbG_m)$ does not
  depend on $X$ but only on the underlying orbicurve and on the
  generic stabilizer.
\end{abstract}
\maketitle
\section{Introduction}
A classical theorem of Tsen states that a function field $K$ of
dimension $1$ over an algebraically closed field is
quasi-algebraically closed (Theorem~\ref{theorem_qac_tsen}). As a
consequence, the \'etale cohomology groups $\h^r(\spec K, \bbG_m)$
vanish for $r \geq 1$ (Corollary~\ref{cor_qac}). Combining this result
with \'etale cohomology techniques, Milne showed that, for a connected
smooth curve $C$ over an algebraically closed field, $\h^r(C, \bbG_m)$
vanishes for $r \geq 2$ (\cite{milne}, III.2.22(d)). \'Etale
cohomology in low degree with coefficients in $\bbG_m$ can be
interpreted geometrically. It is well-known for a scheme $X$ that
$\h^1(X, \bbG_m)\cong \pic (X)$ (\cite{milne}, III.4.9). Moreover, for
a smooth variety $X$ over a field, $\h^2(X, \bbG_m)$ is isomorphic to
the Brauer group $\Br (X)$ (\cite{milne}, IV.2.15). It follows that
computing the groups $\h^r(X, \bbG_m)$ contributes to a better
understanding of geometric properties of $X$. As an example, the
vanishing of $\h^2(X, \bbG_m)$ implies that every gerbe over $X$
banded by a finite group (of order not divided by the characteristic
of the base field) is obtained as a finite number of root
constructions (\cite{cadman}). More in general, once we know $\h^r(X,
\bbG_m)$, we can use Kummer sequence to compute $\h^r(X, \mu_n)$, for
all $n$ not divided by the characteristic of the base field.

In this paper we study the groups $\h^r(X, \bbG_m)$ where $X$ is a
connected smooth tame Deligne-Mumford stack of dimension $1$ over an
algebraically closed field $k$. Such an $X$ admits a natural map $X
\rightarrow C$ to its coarse moduli space, which is a connected smooth
curve over $k$; moreover the map $X \rightarrow C$ factors via an
\'etale gerbe $X \rightarrow Y$, where $Y$ is
an orbicurve (\cite{AOV}, Appendix~A).

We generalize to algebraic stacks the \'etale cohomology techniques
used by Milne and we obtain a description of the groups $\h^r(X,
\bbG_m)$ (Theorem~\ref{theorem_cohom}). In particular, we show that if
the stabilizers are abelian then $\h^2(X, \bbG_m)$ depends only on the
orbicurve $Y$ and the generic stabilizer, not on the structure
of the gerbe $X \rightarrow Y$ (Proposition~\ref{cor_h2}). This result
suggests to investigate two special cases: when $X=Y$ is an orbicurve
(Corollary~\ref{cor_orbicurve}) and when the gerbe $X \rightarrow Y$
is trivial (Proposition~\ref{prop_trivial_gerbe}).
The cohomology of orbicurves, and more in general of twisted nodal
curves, has been described in \cite{chiodo}, Theorem~3.2.3. We include
it for completeness, since we give a different proof
(Proposition~\ref{prop_twisted}). We show that, for a twisted nodal
curve $Y$, the vanishing of $\h^2(Y, \bbG_m)$ implies that every
$G$-gerbe over $Y$ banded by a finite group $G$ (of order not divided
by $\Char k$), can be obtained as a finite number of root
constructions (Section~\ref{subsection_abeliangerbes}). This fact has
been used in \cite{andreini} and \cite{chiodo} to study the moduli
stack of twisted stable maps and the associated Gromov-Witten
invariants.
Finally, we show with two examples that, in general, the higher
cohomology groups $\h^r(X, \bbG_m)$ cannot be computed knowing only
the base of the gerbe $X \rightarrow Y$ and the banding group
(Section~\ref{section_example}).
\subsection*{Acknoledgements}
I would like to thank my Ph.D. advisor, Barbara Fantechi, for
suggesting me this problem and for all the helpful discussions and
suggestions. I'm grateful to Angelo Vistoli for teaching me the key
techniques used here. I would also like to thank Andrew Kresch and
Matthieu Romagny for corrections and suggestions.
\subsection*{Notations}
We write $\Br (K)$ for the Brauer group of a field $K$. With the word
stack we always mean a Deligne-Mumford algebraic stack in the sense of
\cite{DM}. All stacks are assumed to be separated and of finite type
over the base field. For the notion of tame stack see \cite{AOV}. An
orbicurve is a stack of dimension $1$ with trivial generic
stabilizer. We denote by $\scrO_{S,\barx}^{\sh}$ the strictly
Henselian local ring of a scheme (or an algebraic stack) $S$ at a
geometric point $\barx \rightarrow S$. We denote by $\mu_n$ the sheaf
defined by the group scheme $\spec \sfrac{\bbZ[t]}{(t^n -1)}$.
\section{Preliminaries}\label{section_prem}
We recall a few classical theorems about quasi-algebraically closed
fields and cohomology of finite abelian groups (for more details see
\cite{shatz}, chapter IV, and \cite{weibel}, chapter VI).
\begin{definiz}
A field $K$ is quasi-algebraically closed if every non constant
homogeneous polynomial $f(x_1, \ldots, x_N)$ of degree $d<N$ with
coefficients in $K$ has a non-trivial zero in $K$.
\end{definiz}
\begin{prop}[\cite{shatz}, IV.3 Corollary~1]\label{prop_qac}
Let $K$ be a quasi-algebraically closed field and let $G$ denote the
absolute Galois group $\gal(\sfrac{K_{\s}}{K})$, where $K_{\s}$ is a
separable closure of $K$. Then
\begin{enumerate}
\item $\Br (K) =0$;
\item $H^r(G,T)=0$ for $r \geq 2$ and for all discrete torsion
  $G$-modules $T$;
\item $H^r(G,M)=0$ for $r \geq 3$ and for all discrete $G$-modules $M$.
\end{enumerate}
\end{prop}
\begin{cor}\label{cor_qac}
Let $K$ be a quasi-algebraically closed field. Then $\h^r(\spec K,
\bbG_m)=0$ for $r \geq 1$.
\end{cor}
\begin{proof}
By \cite{milne}, III.1.7 and Proposition~\ref{prop_qac},
\begin{equation*}
\qquad\qquad\qquad\quad\;\,\h^r(\spec K,\bbG_m) =
\begin{cases}
\pic (\spec K) =0 & \text{if }r=1\\
H^2(G, K_{\s}^*)= \Br (K)=0 & \text{if }r=2\\
H^r(G, K_{\s}^*)=0 & \text{if }r \geq 3 \text{.}\,\;\quad\qquad\qquad\qquad\qedhere
\end{cases}
\end{equation*}
\end{proof}
\begin{teorema}[\cite{shatz}, IV.3 Theorem~24]\label{theorem_qac_tsen}
Let $K$ be a function field of dimension $1$ over an algebraically
closed field $k$. Then $K$ is quasi-algebraically closed.
\end{teorema}
\begin{teorema}[\cite{shatz}, IV.3 Theorem~27]\label{theorem_qac_lang}
Let $K$ be the field of fractions of an Henselian discrete valuation
ring $R$ with algebraically closed residue field. Let $\hat{K}$ be the
completion of $K$, and assume that $K \subset \hat{K}$ is a separable
field extension. Then $K$ is quasi-algebraically closed.
\end{teorema}
\begin{oss}
The separability condition in Theorem~\ref{theorem_qac_lang} holds in
the case $R= \scrO_{X,\barx}^{\sh}$, with $X$ a scheme of finite type
over a field (\cite{milne}, III.2.22(b)).
\end{oss}
\begin{teorema}[\cite{weibel}, Theorem~6.2.2]\label{prop_cohom_groups}
Let $G=\sfrac{\bbZ}{d \bbZ}$ and let $A$ be a discrete
$G$-module. Then
\begin{equation*}
H^r(G,A)=
\begin{cases}
A^G & r=0\\
\sfrac{\ker \N_G}{I_G A} & r \equiv 1 \pod{2}\\
\sfrac{A^G}{\im \N_G} & r \equiv 0 \pod{2}\text{, }r>0\text{,}
\end{cases}
\end{equation*}
where $\N_G \colon A \rightarrow A$ is the norm $\N_G(a) = \sum_{g \in
  G}ga$, $A^G$ is the set of elements in $A$ fixed by the $G$-action,
and $I_GA$ is the subgroup of $A$ generated by elements $(ga-a)$, with
$a \in A$ and $g \in G$.
\end{teorema}
\begin{cor}\label{cor_Z}
Let $G$ be a finite cyclic group acting trivially on $\bbZ$, then
\begin{equation*}
H^r(G, \bbZ)=
\begin{cases}
\bbZ & r=0\\
0 & r \equiv 1 \pod{2}\\
G & r \equiv 0 \pod{2}\text{, }r>0\text{.}
\end{cases}
\end{equation*}
\end{cor}
\begin{lemma}\label{lemma_zero_maps}
  Let $G$ be a finite group and let $(J^\bullet,d)$ be a cochain
  complex of abelian groups on which $G$ acts trivially. Then, for $p
  \geq 0$, $q \geq 1$ and $r=2, \ldots, q+1$, the following map is
  zero \[\Phi_r^{p,q} \colon H^p(G, H^q(J^\bullet)) \rightarrow
  H^{p+r}(G,H^{q-r+1} (J^\bullet))\text{.}\]
\end{lemma}
\begin{proof}
Let assume $r=2$.  The map $\Phi_2^{p,q}$ is defined as follows. Let
$\xi \in H^p(G, H^q(J^\bullet))$, then $\xi$ is the class of a
cycle $f \colon G^p \rightarrow H^q(J^\bullet)$. Let
$\tilde{f}\colon G^p \rightarrow Z^q(J^\bullet)$ be a lifting of $f$
(we write $Z^q(J^\bullet)$ for cycles in $J^q$). If we denote by
$d_G$ the boundary map for the group cohomology, then we have
$d_G(f)=0$, hence $d_G(\tilde{f}) \colon G^{p+1} \rightarrow d(J^{q-1})$.
In particular, there exists a map $\nu_f \colon G^{p+1} \rightarrow
J^{q-1}$ such that $d \circ \nu_f =d_G(\tilde{f})$. Then $\Phi_2^{p,q} (\xi)= \pi
\circ d_G(\nu_f)$, where $\pi \colon Z^q(J^\bullet) \rightarrow
H^q(J^\bullet)$ is the projection.
It is enough to show that $d_G(\tilde{f})=0$, because
then $\im \nu_f \subset Z^{q-1}(J^\bullet)$ and
therefore $\pi \circ d_G(\nu_f)= d_G(\pi \circ \nu_f)=0$.
Since the action of $G$ is trivial, for all $(g_1, \ldots, g_{p+1}) \in G^{p+1}$,
\[d_G(\tilde{f})(g_1, \ldots, g_{p+1})\!\!=\!\!\tilde{f}(g_2, \ldots,
g_{p+1}) +\!\sum_{i=1}^p\!{(-1)}^i\tilde{f}(g_1, \ldots, g_ig_{i+1},
\ldots, g_{p+1})\!+\!{(-1)}^{p+1} \tilde{f}(g_1, \ldots, g_p)\text{.}\]
If we require $d_G(\tilde{f})=0$, then we get a linear system of $n^{p+1}$ equations
of the form
\begin{align}\label{eq:dG}
x_{(g_2, \ldots,
g_{p+1})} + \sum_{i=1}^p {(-1)}^ix_{(g_1, \ldots, g_ig_{i+1},
\ldots, g_{p+1})} + {(-1)}^{p+1} x_{(g_1, \ldots, g_p)} = 0\text{,}
\end{align}
in the indeterminates $\insieme{x_{\underline{g}}}{\underline{g}=(g_1,
  \ldots, g_p) \in G^p}$, where $n$ is the order of $G$. Since
$d_G(f)=0$, there exists a solution of these
equations in $H^q(J^\bullet)$. In particular there is a set $I$ of indeterminates whose
values can be chosen freely, and the values of the others are
determined by~\eqref{eq:dG}. If $I = \emptyset$ then there is only the
trivial solution. Notice that $\pi$ preserves
relations~\eqref{eq:dG}. Hence we can assign values
$z_{\underline{g}}\in Z^p(J^\bullet)$ to the indeterminates in $I$ so
that $\pi(z_{\underline{g}})=f(\underline{g})$. Then,
using equations~\eqref{eq:dG}, we find a solution $z_{\underline{g}}\in Z^p(J^\bullet)$ of the system~\eqref{eq:dG}. It follows
that the map $\tilde{f}$ defined by $\tilde{f}(\underline{g})= z_{\underline{g}}$ is a lifting of
 $f$ such that $d_G(\tilde{f})=0$. Hence $\Phi_2^{p,q}=0$, for all $p
 \geq 0$ and $q \geq 1$.

The maps $\Phi_r^{p,q}$ are defined recursively.
Let $\xi \in
H^p(G, H^q(J^\bullet))$, and consider $\Phi_r^{p,q}(\xi)$. Let $\phi
\colon G^{p+r} \rightarrow H^{q-r+1}(J^\bullet)$ be a cycle that
represents $\Phi_r^{p,q}(\xi)$. We know that \[\Phi_2^{p+r,
  q-r+1}(\Phi_r^{p,q}(\xi)) = d_G(\pi \circ \nu_\phi)=0\text{,}\] with
$\nu_\phi$ defined as above, then we set $\Phi_{r+1}^{p,q}(\xi)= \pi \circ
\nu_\phi$.
Now we prove the statement by induction on $r$. Assume that
$\Phi_r^{p,q}=0$.  Let $\xi \in H^p(G, H^q(J^\bullet))$, then $\phi =
\Phi_r^{p,q}(\xi)=0$. It follows that there exists $\rho \colon
G^{p+r} \rightarrow Z^{q-r}(J^\bullet)$ which is a lifting of
$\phi$. Hence we can take $\nu_\phi = d_G(\rho)$ and we have \[\pi
\circ \nu_\phi = \pi \circ d_G(\rho) = d_G(\pi \circ \rho) =
0\text{.}\qedhere\]
\end{proof}
\section{Setting}\label{section_setting}
Let $X$ be a connected smooth tame Deligne-Mumford stack of dimension
$1$ over an algebraically closed field $k$.  Let $\eta$ be the generic
point of $X$ and let $g \colon \eta \rightarrow X$ be the
inclusion. We denote by $G_0$ the generic stabilizer. Moreover if
$\sigma$ is a closed point of $X$ with stabilizer $G_\sigma$, we write $\sigma
\colon \sigma \hookrightarrow X$ for the inclusion and we denote by
$X(k)$ the set of closed points of $X$.  Let $C$ be the coarse moduli
space of $X$, we denote by $\pi \colon X \rightarrow C$ the natural
map.

We begin with some general remarks on the stack $X$ and its coarse
moduli space.
\begin{prop}\label{prop_cms}
The coarse moduli space $C$ is a connected smooth curve over $k$.
\end{prop}
\begin{proof}
  Since $X$ is connected and $\pi$ is surjective, also $C$ is
  connected. Moreover, there exists an \'etale cover of $C$ consisting
  of schemes of the form $\sfrac{U}{G}$, where $U$ is a smooth affine
  scheme of dimension $1$ over $k$ and $G$ is a finite group of order
  not divided by $\Char k$ (\cite{AOV},
  Theorem~3.2). By \cite{milne}, Proposition~I.3.24, locally in the
  \'etale topology, there exists a linearization of the action of $G$
  on $U$. In particular $G \subset k^*$ is a finite subgroup, hence $G
  = \mu_n$. It follows that, locally in the \'etale topology,
  $\sfrac{U}{G}=\spec {(k [t])}^{\mu_n}= \spec k[t^n]$, which is
  smooth of dimension $1$.
\end{proof}
\begin{teorema}[\cite{milne}, III.2.22(d)]\label{teorema_milne}
Let $C$ a smooth curve over an algebraically closed
field. Then
\begin{equation*}
\h^r(C, \bbG_m) =
\begin{cases}
\sez (C, \bbG_m) & r=0\\
\pic (C) & r=1\\
0 & r \geq 2\text{.}
\end{cases}
\end{equation*}
\end{teorema}
\begin{teorema}[\cite{AOV}, A.1]\label{theorem_AOV}
  Let $M$ be a regular Deligne-Mumford stack over a field $k$. Let
  $\scrG$ be the closure of the fibers of the inertia stack $I(M)$
  over the generic points of irreducible components of the moduli
  space of $M$. Then there exist a regular Deligne-Mumford stack $Y$
  with trivial generic stabilizer and an \'etale morphism $M
  \xrightarrow{\pi} Y$, such that $M$ is a $\scrG$-gerbe banded by
  $\scrG$ over $Y$.
\end{teorema}
\paragraph{}
By Theorem~\ref{theorem_AOV}, $X$ is an \'etale gerbe over an
orbicurve $Y$.  It is known that $Y$ is obtained from its coarse
moduli space $C$ by applying a finite number of root
constructions. Explicitly, there exist distinct points $p_1, \ldots,
p_N \in C(k)$ and integers $d_1, \ldots, d_N \geq 2$, with  $N\geq 0$, such that $X =
X_1 \times_C \cdots \times_C X_N$, where $X_l =
\sqrt[d_l]{\sfrac{p_l}{C}}$, for $l=1, \ldots, N$ (for details on the
root construction see \cite{cadman}).  Let $\Sigma = \{\sigma_1,
\ldots, \sigma_N\}$ be the set of closed points of $X$ corresponding
to $p_1, \ldots, p_N \in C(k)$. Then $G_\sigma = G_0$ for $\sigma \in
X(k)$, $\sigma \notin \Sigma$, and $\sfrac{G_{\sigma_l}}{G_0} =
\sfrac{\bbZ}{d_l \bbZ}$ for $l=1, \ldots, N$.
\begin{prop}\label{prop_genericgerbe}
Let $\spec K$ be the generic point of $C$, then $\eta=\spec K \times \B G_0$.
\end{prop}
\begin{proof}
By Theorem~\ref{theorem_AOV}, the stack $X$ is a gerbe over $Y$ banded
by $\scrG$. Let $\scrG_\eta$ be the sheaf over $\spec K$ induced by
$\scrG$. Notice that $\eta = \spec K \times_Y X$, hence $\eta
\rightarrow \spec K$ is a $\scrG_\eta$-gerbe banded by
$\scrG_\eta$. Recall that such gerbes are classified by $\h^2(\spec K,
\scrZ)$, where $\scrZ$ is the center of $\scrG_\eta$ (\cite{giraud},
IV.5.2). Let $\gal(\sfrac{K_s}{K})$ be the absolute Galois group of
$K$, then there is an isomorphism \[\h^r(\spec K, \scrZ) \cong
H^r(\gal(\sfrac{K_s}{K}), M)\text{,}\] with $M = \varinjlim_L
\scrZ(\spec L)$, where the limit is taken over all subfields $L\subset
K_s$ that are finite over $K$ (\cite{milne}, III.1.7). Notice that $M$
is a torsion module, hence, by
Proposition~\ref{prop_qac}, $H^r(\gal(\sfrac{K_s}{K}), M) =
0$ for $r \geq 2$. In particular $\h^2(\spec K, \scrZ)=0$.
\end{proof}
\begin{oss}\label{remark_trivialaction}
The action of $G_0$ over $K^*$ is trivial, since the sheaf
$\bbG_{m,\eta}$ comes from the sheaf $\bbG_{m, \spec K}$ (which is
equivariant under the action of $G_0$). It follows that ${\pi_\eta}_*
\bbG_{m,\eta}= \bbG_{m, \spec K}$, where $\pi_\eta \colon \eta
\rightarrow \spec K$ is the natural morphism.
Similarly, for $\sigma \in X(k)$, the action of $G_\sigma$ over $\bbZ$
is trivial and ${\pi_\sigma}_* \bbZ = \bbZ$, where
$\pi_\sigma \colon \sigma \rightarrow \spec k$.
\end{oss}
\section{The Weil-divisor exact sequence}
Throughout this section we will use the notations of
Section~\ref{section_setting}.
\begin{prop}\label{prop_succ_fasci}
There exists an exact sequence of sheaves on $X$ (with the \'etale topology)
\begin{equation}\label{eq:succ_fasci}
0 \rightarrow \bbG_{m,X} \xrightarrow{\alpha}
g_*\bbG_{m,\eta} \xrightarrow {\beta} \bigoplus_{\sigma \in
  X(k)}{\sigma}_*\bbZ_\sigma \rightarrow
0\text{;}
\end{equation}
$\alpha$ and $\beta$ are defined as follows: if $W
\xrightarrow{f} X$ is an étale morphism from a connected scheme and
$R(W)$ is the ring of rational functions on $W$, then $\sez (W,
\bbG_m) \xrightarrow{\alpha(W,f)} {R(W)}^*$ is induced by the morphism
$W\times_X \eta \rightarrow W$ and, for all $\lambda \in
{R(W)}^*$, \[\beta(W,f) (\lambda) = \sum_{w \in W(k)}
\val_w(\lambda)\text{,}\] where $\val_w$ is the discrete valuation on
$R(W)$ induced by $\scrO_{W,w}$.
\end{prop}
\begin{proof}
If $W$ is a scheme as in the statement, then it is
normal and regular.
Moreover
\begin{align*}
\bbG_{m,X}(W,f) &= \sez (W,\bbG_m)\\
g_*\bbG_{m,\eta}(W,f) &= \sez (W \times_X \eta,\bbG_m) = {R(W)}^*\\
\bigoplus_{\sigma \in X(k)}\sigma_*\bbZ_\sigma (W,f) &= \bigoplus_{\sigma \in X(k)} \sez (W \times_X \sigma, \bbZ) = \bigoplus_{w \in W(k)} \bbZ \text{.}
\end{align*}
Therefore $\alpha(W,f)$ and $\beta(W,f)$ are well-defined. If $Y
\xrightarrow{g} X$ is an other étale morphism from a scheme, we form
the fiber product $h \colon Y \times_X W \rightarrow X$. Then the maps
induced by restrictions on $Y \times_X W$ by $\alpha(W,f)$ and
$\alpha(Y,g)$ coincide with $\alpha(Y \times_X W, h)$. A similar
argument holds for $\beta$. It follows that $\alpha$ and $\beta$
are well-defined.

Recall that the exactness of a sequence of sheaves on a stack can be
checked on stalks at geometric points.
Let $\barx \xrightarrow{x} X$ be a geometric point and let $f \colon
U\rightarrow X$ be an \'etale morphism from a smooth connected
curve, such that $x$ farctors through $f$. Then
\begin{align*}
{(\bbG_{m,X})}_{\barx}&={(\bbG_{m,U})}_{\barx}= A^*\text{,}\\
{(g_*\bbG_{m,\eta})}_{\barx} &={({g_U}_*\bbG_{m,\eta_U})}_{\barx} = {Q(A)}^*\text{,}
\end{align*}
where $A = \scrO_{U,\barx}^{\sh}$ (note that $A$ is a discrete
valuation ring), $Q(A)$ is its field of fractions, the morphism $g_U
\colon \eta_U \rightarrow U$ is the inclusion of the generic point of
$U$ (since $g$ is an open immersion and $U$ is irreducible, we obtain
$U \times_X \eta = \eta_U$). Moreover, for every $\sigma \in X(k)$,
$\sigma \times_X U$ is a set of closed points of $U$ whose image in
$X$ is $\sigma$.  Hence $(\sigma_*\bbZ_\sigma) =\bigoplus_{u \in
  \sigma \times_X U}u_*\bbZ$ and ${(u_*\bbZ)}_{\barx}$ is non zero if
and only if $u=x$, in which case ${(u_*\bbZ)}_{\overline{u}}=\bbZ$.
Therefore the sequence of stalks is \[0 \rightarrow A^* \rightarrow
{Q(A)}^* \xrightarrow{\val_A} \bbZ \rightarrow 0 \text{,}\] where
$\val_A$ is the discrete valuation on $Q(A)$ defined by $A$; this
sequence is exact by \cite{milne}, II.3.9.
\end{proof}
\paragraph{}
The short exact sequence~\eqref{eq:succ_fasci} induces a long exact
sequence of \'etale cohomology groups
\begin{multline}\label{eq:succ_cohom}
\cdots \rightarrow \h^r(X, \bbG_{m,X}) \rightarrow \h^r(X, g_*\bbG_{m,\eta}) \rightarrow \bigoplus_{\sigma \in X(k)}\h^r(X, \sigma_*\bbZ_\sigma) \rightarrow \h^{r+1}(X, \bbG_{m,X}) \rightarrow \cdots
\end{multline}
where we used the fact that \'etale cohomology commutes with arbitrary direct
sums of sheaves on a quasi-compact Deligne-Mumford stack (see
\cite{milne}, chapter III, Remark~3.6(d)).
\begin{oss}
By \cite{milne}, III.2.22, there is a short exact sequence in cohomology
\begin{equation}\label{eq:succ_C}
0 \rightarrow \sez (C, \bbG_m) \rightarrow K^* \xrightarrow{\beta_0}
\bigoplus_{x \in C(k)} \bbZ \rightarrow \pic (C) \rightarrow 0\text{,}
\end{equation}
obtained from the Weil-divisor exact sequence for $C$ (\cite{milne}, II.3.9).
\end{oss}
\begin{prop}\label{prop_alg_closed}
Let $L$ be a quasi-algebraically closed field and let $G$ be a finite
group of order not divided by the characteristic of $L$, acting
trivially on $L$. Let $F$ be a sheaf over the quotient stack
$[\sfrac{\spec L}{G}]$. If either $L$ is algebraically closed or $F=
\bbG_m$ then, for $r\geq 0$, \[\h^r([\sfrac{\spec L}{G}], F) \cong
H^r(G, \sez (\spec L, F))\text{.}\]
\end{prop}
\begin{proof}
The natural map $\spec L \rightarrow [\sfrac{\spec L}{G}]$ is a Galois
covering with group $G$, then we can consider the Hochschild-Serre
spectral sequence \[E_2^{p,q}=H^p(G, \h^q(\spec L, F)) \Rightarrow
\h^{p+q}([\sfrac{\spec L}{G}], F)\text{.}\] We can view the groups
$\h^r(\spec L, F)$ as Galois cohomology groups of the absolute Galois
group of $L$ (\cite{milne}, III.1.7). If $L$ is algebraically closed
then $\h^r(\spec L, F) =0$ for $r \geq 1$. If $F=\bbG_m$ then, by
Corollary~\ref{cor_qac}, $\h^r(\spec L, \bbG_m) =0$ for $r \geq 1$. In
both cases the sequence degenerates and we get the statement.
\end{proof}
\begin{lemma}\label{lemma_closedpoint}
  Let $\sigma \in X(k)$ with stabilizer $G_\sigma$. Then $\h^r(X,
  \sigma_* \bbZ_{\sigma}) = H^r(G_\sigma, \bbZ)$, for $r\geq 0$.
\end{lemma}
\begin{proof}
Consider the Leray spectral sequence for the inclusion $\sigma
\xrightarrow{\sigma} X$, \[E_2^{p,q}=\h^p(X, R^q \sigma_*
\bbZ_{\sigma}) \Rightarrow \h^{p+q}(\sigma, \bbZ) \text{.}\] Since
$\sigma$ is a closed embedding, the functor $\sigma_*$ is exact
(\cite{milne}, II.3.6), hence $R^q \sigma_* \bbZ_{\sigma} = 0$ for $q
\geq 1$. Therefore the spectral sequence degenerates and
$\h^r(X, \sigma_* \bbZ_{\sigma}) \cong \h^r(\sigma, \bbZ)$ for $r \geq
0$. By Theorem~\ref{theorem_AOV}, $\sigma$ is a gerbe over $\spec k$
banded by $\scrG_\sigma$, and recall that such gerbes are classified by
$\h^2(\spec k, \scrG_\sigma)$, which vanishes since $k$ is algebraically
closed. It follows that $\sigma$ is the trivial $G_\sigma$-gerbe over
$\spec k$ and, by Proposition~\ref{prop_alg_closed}, $\h^r(\sigma,
\bbZ) \cong H^r(G_\sigma, \bbZ)$ for $r \geq 0$.
\end{proof}
\begin{lemma}\label{lemma_genericpoint}
  Let $G_0$ denote the stabilizer of $\eta$. Then $\h^r(X, g_*
  \bbG_{m,\eta}) = H^r(G_0, K^*)$, for $r \geq 0$.
\end{lemma}
\begin{proof}
Consider the Leray spectral sequence for $\eta \xrightarrow{g} X$, \[E_2^{p,q}=\h^p(X, R^q
g_* \bbG_{m, \eta}) \Rightarrow \h^{p+q}(\eta, \bbG_m)\text{.}\] Let
$\barx \xrightarrow{x} X$ be a geometric point and let $f \colon
U\rightarrow X$ be an \'etale morphism from a connected
curve, such that $\barx$ farctors through $f$. Then $\eta \times_X U$
is the generic point $\eta_U$ of $U$. Therefore \[{(R^q g_*\bbG_{m, \eta})}_{\barx} = {(R^q{g_U}_*\bbG_{m,
    \eta_U})}_{\barx} = \h^q(\spec \scrO_{U,\barx}^{\sh}\times_U \eta_U,
\bbG_m)= \h^q(\spec Q(\scrO_{U,\barx}^{\sh}), \bbG_m) = 0
\] for $q \geq 1$, where $Q(\scrO_{U,\barx}^{\sh})$ is the field of fractions of $\scrO_{U,\barx}^{\sh}$ and the last equality follows from the fact that
$Q(\scrO_{U,\barx}^{\sh})$ is quasi-algebraically closed
(Theorem~\ref{theorem_qac_lang}). Hence $R^qg_* \bbG_{m, \eta} = 0$
for $q \geq 1$ and $\h^r(X, g_*\bbG_{m, \eta}) \cong \h^r(\eta,
\bbG_m)$ for $r \geq 0$. Since $\eta = [\sfrac{\spec K}{G_0}]$ is the
trivial gerbe over $\spec K$ (Proposition~\ref{prop_genericgerbe}),
the statement follows by Proposition~\ref{prop_alg_closed}.
\end{proof}
\paragraph{}
By Lemma~\ref{lemma_closedpoint} and Lemma~\ref{lemma_genericpoint},
the sequence~\eqref{eq:succ_cohom} becomes
\begin{multline}\label{eq:succ_beta}
\cdots \rightarrow \h^{r}(X, \bbG_m) \rightarrow H^r(G_0, K^*)
\xrightarrow{\beta^{(r)}} \bigoplus_{\sigma\in
  X(k)}H^r(G_\sigma,\bbZ)\rightarrow \h^{r+1}(X, \bbG_m)
\rightarrow \cdots \text{.}
\end{multline}
\begin{lemma}\label{lemma_beta}
For every $r\geq 0$, the map $\beta^{(r)}$ in \eqref{eq:succ_beta}
factors as \[H^r(G_0,K^*) \xrightarrow{\beta_0^{(r)}}\bigoplus_{\sigma
  \in X(k)}H^r(G_0,\bbZ) \xrightarrow{\tau^{(r)}} \bigoplus_{\sigma
  \in X(k)}H^r(G_\sigma,\bbZ)\text{,}\] where $\beta_0^{(r)}$ is
induced by $\beta_0$ in \eqref{eq:succ_C} and $\tau^{(r)}$ is the
transfer map (\cite{weibel}, 6.7.16) on each component.
\end{lemma}
\begin{proof}
Let us notice that $H^1(G_\sigma, \bbZ)= \Hom (G_\sigma, \bbZ)=0$,
since $G_\sigma$ acts trivially on $\bbZ$ and $\bbZ$ does not contain
non-trivial finite subgroups. From \eqref{eq:succ_C} and
\eqref{eq:succ_beta}, we get the following commutative
diagram with exact rows (by Remark~\ref{remark_trivialaction}, $\pi^* \bbG_{m,C} = \bbG_{m,X}$
and $\pi_\sigma^* \bbZ = \bbZ$)
  \[
  \begin{tikzpicture}
    \def\x{2.4}
    \def\y{-1.6}
    \node (A1_0) at (0*\x, 1*\y) {$0$};
    \node (A1_1) at (0.9*\x, 1*\y) {$\sez (C, \bbG_m)$};
    \node (A1_2) at (1.8*\x, 1*\y) {$K^*$};
    \node (A1_3) at (2.7*\x, 1*\y) {$\displaystyle{\bigoplus_{\sigma \in X(k)}\bbZ}$};
    \node (A1_4) at (3.8*\x, 1*\y) {$\pic (C)$};
    \node (A1_5) at (5*\x, 1*\y) {$0$};
    \node (A2_0) at (0*\x, 2*\y) {$0$};
    \node (A2_1) at (0.9*\x, 2*\y) {$\h^0(X, \bbG_m)$};
    \node (A2_2) at (1.8*\x, 2*\y) {$K^*$};
    \node (A2_3) at (2.7*\x, 2*\y) {$\displaystyle{\bigoplus_{\sigma \in X(k)}\bbZ}$};
    \node (A2_4) at (3.8*\x, 2*\y) {$\h^1(X, \bbG_m)$};
    \node (A2_5) at (5*\x, 2*\y) {$H^1(G_0, K^*)$};
    \node (A2_6) at (5.9*\x, 2*\y) {$0$};
    \path (A1_3) edge [->] node [auto] {$\scriptstyle{{(\pi_\sigma^*)}_\sigma}$} (A2_3);
    \path (A1_5) edge [->] node [auto] {$\scriptstyle{}$} (A2_5);
    \path (A2_0) edge [->] node [auto] {$\scriptstyle{}$} (A2_1);
    \path (A1_1) edge [->] node [auto] {$\scriptstyle{\pi^*}$} (A2_1);
    \path (A2_4) edge [->] node [auto] {$\scriptstyle{}$} (A2_5);
    \path (A1_2) edge [->] node [auto] {$\scriptstyle{\pi_\eta^*}$} (A2_2);
    \path (A2_2) edge [->] node [auto] {$\scriptstyle{\beta}$} (A2_3);
    \path (A1_4) edge [->] node [auto] {$\scriptstyle{}$} (A1_5);
    \path (A2_1) edge [->] node [auto] {$\scriptstyle{}$} (A2_2);
    \path (A2_3) edge [->] node [auto] {$\scriptstyle{}$} (A2_4);
    \path (A1_2) edge [->] node [auto] {$\scriptstyle{\beta_0}$} (A1_3);
    \path (A1_4) edge [->] node [auto] {$\scriptstyle{\pi^*}$} (A2_4);
    \path (A1_0) edge [->] node [auto] {$\scriptstyle{}$} (A1_1);
    \path (A1_1) edge [->] node [auto] {$\scriptstyle{}$} (A1_2);
    \path (A1_3) edge [->] node [auto] {$\scriptstyle{}$} (A1_4);
    \path (A2_5) edge [->] node [auto] {$\scriptstyle{}$} (A2_6);
  \end{tikzpicture}
  \]
Elements in $\sez(\spec K, \bbG_m)$ corresponds to rational maps $C
\rightarrow \bbP^1$. Since $C$ is smooth, we obtain $\sez(\spec K, \bbG_m)
\subset \Mor(C, \bbP^1)$. Similarly, $\h^0(\eta, \bbG_m) \subset
\Mor(X, \bbP^1)$. Because $C$ is the coarse moduli space, $\Mor(C,
\bbP^1) \cong \Mor(X, \bbP^1)$ and under this identification
$\pi_\eta^*= \id$.

Recall that $\pi$ factors through an \'etale gerbe $X \rightarrow Y$
over an orbicurve $Y$. Let $\beta_Y$ be the analogous of $\beta$ for
$Y$, then $\beta = \beta_Y$, since the valuation at a point doesn't
change for \'etale morphisms. Hence we can assume $X=Y$. As in the
proof of Proposition~\ref{prop_cms}, \'etale locally
$Y=[\sfrac{\spec k[t]}{\mu_r}]$ and $C = \spec k[t^r]$. Considering
the \'etale cover $\spec k[t] \rightarrow Y$, we see that
$\pi_\sigma^*$ is the multiplication by the order of the stabilizer of
$\sigma \in Y(k)$ in $Y$. Let $d_\sigma$ be the order of $G_\sigma$
and $d$ the order of $G_0$, then $\beta =
{(\sfrac{d_\sigma}{d})}_{\sigma} \circ \beta_0$ and $\beta^{(r)}$ is
induced by $\beta_0$ and the inclusion $G_0 \hookrightarrow G_\sigma$.
\end{proof}
\begin{oss}\label{remark_beta}
  By Lemma~\ref{lemma_beta} and snake lemma, we have an exact
  sequence \[0 \rightarrow \ker \beta_0^{(r)} \rightarrow \ker
  \beta^{(r)} \rightarrow \ker \tau^{(r)} \rightarrow \coker
  \beta_0^{(r)} \rightarrow \coker \beta^{(r)} \rightarrow \coker
  \tau^{(r)} \rightarrow 0\text{.}\] Moreover, the
  sequence~\eqref{eq:succ_C} induces exact sequences in cohomology
  (notice that $\im \beta_0$ is a free $\bbZ$-module)
\begin{equation*}
\begin{cases}
  0 \rightarrow H^r(G_0,\sez (C, \bbG_m)) \rightarrow H^r(G_0,K^*)
  \xrightarrow{\widetilde{\beta}_0^{(r)}} H^r(G_0,\im
  \beta_0)\rightarrow 0\text{,}&\\
  \cdots \rightarrow H^r(G_0,\im \beta_0) \rightarrow \bigoplus_{x \in
    C(k)}H^r(G_0,\bbZ) \rightarrow H^r(G_0, \pic (C))\rightarrow
  \cdots\text{.}&
\end{cases}
\end{equation*}
By snake lemma, we get $\coker \beta_0^{(r)} \subset H^r(G_0, \pic
(C))$ and $H^r(G_0, \sez (C, \bbG_m)) \subset \ker \beta_0^{(r)}$.
\end{oss}
\begin{teorema}\label{theorem_cohom}
With the notations of Section~\ref{section_setting}, we have $\h^0(X,
\bbG_m)= \sez (C, \bbG_m)$ and, for $r \geq 1$, there exists a short
exact sequence \[0 \rightarrow \coker \gamma^{(r-1)} \rightarrow
\h^r(X, \bbG_m) \rightarrow \ker \beta^{(r)} \rightarrow 0\text{,}\]
where $\ker \beta^{(r)}$ fits in the following exact sequence \[0
\rightarrow H^r(G_0, \sez (C, \bbG_m)) \rightarrow \ker \beta^{(r)}
\rightarrow \ker \gamma^{(r)} \rightarrow 0\] and $\gamma^{(r)}\colon
H^r(G_0, \im \beta_0) \rightarrow \bigoplus_{\sigma \in
  X(k)}H^r(G_\sigma,\bbZ)$ is induced by $G_0 \hookrightarrow
G_\sigma$ and the natural inclusion $\im \beta_0 \subset
\bigoplus_{\sigma \in X(k)}\bbZ$. In particular there is a short exact sequence \[0
\rightarrow \pic(Y) \rightarrow \h^1(X, \bbG_m) \rightarrow \Hom(G_0,
K^*)\rightarrow 0\text{.}\]
\end{teorema}
\begin{proof}
  By sequence~\eqref{eq:succ_beta}, we have short exact sequences \[0
  \rightarrow \coker \beta^{(r-1)}\rightarrow \h^r(X,
  \bbG_m)\rightarrow \ker \beta^{(r)}\rightarrow 0\] for $r\geq 1$.
  By Remark~\ref{remark_beta}, $\h^0(X, \bbG_m) = \h^0(C, \bbG_m)$ and $\beta^{(r)}=
  \gamma^{(r)}\circ \widetilde{\beta}_0^{(r)}$, with $\gamma^{(r)}$ as
  described in the statement. Applying snake lemma we get $\coker
  \beta^{(r-1)}= \coker \gamma^{(r-1)}$ and \[0 \rightarrow H^r(G_0, \sez
  (C, \bbG_m)) \rightarrow \ker \beta^{(r)} \rightarrow \ker
  \gamma^{r} \rightarrow 0\text{.}\] As noticed in the proof of
  Lemma~\ref{lemma_beta}, $\coker \beta = \pic (Y)$, hence we
  obtain \[0 \rightarrow \pic (Y) \rightarrow \h^1(X, \bbG_m)
  \rightarrow \Hom (G_0, K^*)\rightarrow 0\text{,}\] where we used
  that $H^1(G_\sigma, \bbZ)= \Hom(G_\sigma, \bbZ)=0$, since $G_\sigma$
  acts trivially on $\bbZ$.
\end{proof}
\begin{oss}
In the case $G_0$ is abelian and $X$ is a $G_0$-gerbe banded over $Y$,
we recover the description of $\h^1(X, \bbG_m)$ given in \cite{cadman},
Corollary~3.2.1 (see also Section~\ref{subsection_abeliangerbes}).
\end{oss}
\begin{prop}\label{cor_h2}
  If $G_\sigma$ is abelian for every $\sigma \in X(k)$ then $\ker
  \beta^{(r)}$ depends only on $Y$ and $G_0$, not on the structure of
  the gerbe $X \rightarrow Y$. In particular $\h^2(X, \bbG_m)= \h^2(Y
  \times \B G_0, \bbG_m)$.
\end{prop}
\begin{proof}
  By Theorem~\ref{theorem_cohom}, $\h^2(X, \bbG_m)= \ker \beta^{(2)}$.
  Consider the restriction map \[\res^{(r)} \colon H^r(G_{\sigma_l},
  \bbZ) \rightarrow H^r(G_0, \bbZ)\text{,}\] where $G_0$ and
  $G_{\sigma_l}$ act trivially on $\bbZ$. By Hochschild-Serre spectral
  sequence (\cite{weibel}, 6.8.2) \[H^p(\sfrac{\bbZ}{d_l\bbZ},
  H^q(G_0, \bbZ))\Rightarrow H^{p+q}(G_{\sigma_l}, \bbZ)\] and
  Lemma~\ref{lemma_zero_maps}, we see that $\res^{(r)}$ is
  surjective. Recall that $\tau^{(r)} \circ \res^{(r)}= d_l$
  (\cite{weibel}, Lemma~6.7.17), where $H^r(G_{\sigma_l}, \bbZ)
  \xrightarrow{d_l} H^r(G_{\sigma_l}, \bbZ)$ is the multiplication by
  $d_l$ which is induced by $\bbZ \xrightarrow{d_l} \bbZ$. It follows
  that $\ker \tau^{(r)}= \im \res^{(r)}|_{\ker d_l}$. The same
  argument applies to $Y\times \B G_0$. Notice
  that \[\tau_Y^{(r)}\colon H^r(G_0, \bbZ)\rightarrow H^r(G_0 \times
  \sfrac{\bbZ}{d_l\bbZ}, \bbZ)\] is induced by composition with
  the projection $G_0 \times \sfrac{\bbZ}{d_l\bbZ}\rightarrow G_0$ and
  the multiplication by $d_l$, therefore $\ker d_l\!\subset\!\ker
  \tau_Y^{(r)}$. As a consequence $\ker \tau^{(r)}\!\!\subset\!\ker
  \tau_Y^{(r)}$ and hence $\ker \beta^{(r)}\!\!\subset\!\ker
  \beta_Y^{(r)}$, where $\beta_Y^{(r)}\!=\!\tau_Y^{(r)}\circ
  \beta_0^{(r)}$. By Remark~\ref{remark_beta}, we get the following
  commutative diagram with exact rows
  \[
  \begin{tikzpicture}[xscale=2.8,yscale=-1.6]
    \node (A0_0) at (0.3, 0) {$0$};
    \node (A0_1) at (1, 0) {$\ker \beta_0^{(r)}$};
    \node (A0_2) at (1.9, 0) {$\ker \beta^{(r)}$};
    \node (A0_3) at (2.8, 0) {$\ker \tau^{(r)}$};
    \node (A0_4) at (4, 0) {$\coker \beta_0^{(r)}$};
    \node (A1_0) at (0.3, 1) {$0$};
    \node (A1_1) at (1, 1) {$\ker \beta_0^{(r)}$};
    \node (A1_2) at (1.9, 1) {$\ker \beta_Y^{(r)}$};
    \node (A1_3) at (2.8, 1) {$\ker \tau_Y^{(r)}$};
    \node (A2_3) at (2.8, 2) {$\displaystyle{\bigoplus_{x \in C(k)}H^r(G_0, \bbZ)}$};
    \node (A2_4) at (4, 2) {$H^r(G_0, \pic (C))$};
    \node (A1_4) at (4, 1) {$\coker \beta_0^{(r)}$};
    \path (A0_0) edge [->]node [auto] {$\scriptstyle{}$} (A0_1);
    \path (A0_1) edge [->]node [auto] {$\scriptstyle{}$} (A0_2);
    \path (A1_0) edge [->]node [auto] {$\scriptstyle{}$} (A1_1);
    \path (A0_3) edge [right hook->]node [auto] {$\scriptstyle{}$} (A1_3);
    \path (A0_2) edge [right hook->] node [auto] {$\scriptstyle{}$} (A1_2);
    \path (A1_3) edge [right hook->] node [auto] {$\scriptstyle{}$} (A2_3);
    \path (A1_4) edge [right hook->] node [auto] {$\scriptstyle{}$} (A2_4);
    \path (A0_3) edge [->]node [auto] {$\scriptstyle{}$} (A0_4);
    \path (A2_3) edge [->]node [auto] {$\scriptstyle{}$} (A2_4);
    \path (A0_4) edge [-,double distance=1.5pt]node [auto] {$\scriptstyle{}$} (A1_4);
    \path (A1_1) edge [->]node [auto] {$\scriptstyle{}$} (A1_2);
    \path (A1_2) edge [->]node [auto] {$\scriptstyle{}$} (A1_3);
    \path (A0_1) edge [-,double distance=1.5pt]node [auto] {$\scriptstyle{}$} (A1_1);
    \path (A0_2) edge [->]node [auto] {$\scriptstyle{}$} (A0_3);
    \path (A1_3) edge [->]node [auto] {$\scriptstyle{}$} (A1_4);
  \end{tikzpicture}
  \]
  which implies $\ker \beta^{(r)}=\ker \beta_Y^{(r)}$.
\end{proof}
\begin{oss}
  In general, the groups $\h^r(X, \bbG_m)$ depend on the gerbe $X
  \rightarrow Y$; if the stabilizers are not abelian, also $\h^2(X,
  \bbG_m)$ may depend on $X \rightarrow Y$ (see
  Section~\ref{section_example}).
\end{oss}
\begin{cor}\label{cor_cyclic}
  If $G_\sigma$ is cyclic for every $\sigma \in X(k)$ then, for every
  $r \geq 1$,
\[\h^{2r}(X, \bbG_m)= \h^2(Y \times \B G_0, \bbG_m)\]
and there is a short exact sequence \[0 \rightarrow H^2(G_0, \pic (Y))
\rightarrow \h^{2r+1}(X, \bbG_m) \rightarrow Q \rightarrow 0\text{,}\] where
$Q$ fits in the following exact sequence \[0 \rightarrow
\bigoplus_{l=1}^N \sfrac{\bbZ}{d_l \bbZ} \rightarrow Q \rightarrow
\Hom (G_0, K^*)=G_0 \rightarrow 0\text{.}\]
\end{cor}
\begin{proof}
  By Corollary~\ref{cor_Z} and Theorem~\ref{prop_cohom_groups}, the
  sequence~\eqref{eq:succ_beta} induces exact sequences for $r\geq
  1$, \[0 \rightarrow \h^{2r}(X, \bbG_m) \rightarrow
  \sfrac{K^*}{{K^*}^d} \xrightarrow{\beta^{(2)}} \bigoplus_{\sigma \in
    X(k)} G_\sigma \rightarrow \h^{2r+1}(X, \bbG_m) \rightarrow G_0
  \rightarrow 0\text{.}\] In particular $\h^{2r}(X, \bbG_m)=\h^2(X,
  \bbG_m)$ and $\h^2(X, \bbG_m)= \h^2(Y
  \times \B G_0, \bbG_m)$, by Proposition~\ref{cor_h2}. Moreover, by Lemma~\ref{lemma_beta}, the
  map $\beta^{(2)}$ factors as \[H^2(G_0, K^*)
  \xrightarrow{\overline{\beta}} \bigoplus_{\sigma \in X(k)} H^2(G_0,
  \bbZ)=G_0 \hookrightarrow \bigoplus_{\sigma \in X(k)}
  G_\sigma=H^2(G_\sigma, \bbZ)\text{,}\] where $\overline{\beta}=
  (\sfrac{d_\sigma}{d})\circ \beta_0^{(2)}$ is the map induced by
  $\beta$. By Remark~\ref{remark_beta}, there is an exact
  sequence \[H^2(G_0, \im \beta) \rightarrow \bigoplus_{\sigma \in
    X(k)}H^2(G_0, \bbZ) \rightarrow H^2(G_0, \pic (Y)) \rightarrow
  H^3(G_0, \im \beta)=0\text{,}\] which implies $\coker
  \overline{\beta}= H^2(G_0, \pic (Y))$. By snake lemma, we get
  the following exact sequence \[0 \rightarrow H^2(G_0, \pic (Y))
  \rightarrow \coker \beta^{(2)} \rightarrow \bigoplus_{l=1}^N
  \sfrac{\bbZ}{d_l \bbZ} \rightarrow 0\text{.}\qedhere\]
\end{proof}
\begin{cor}\label{cor_orbicurve}
  Let $Y$ be a smooth tame orbicurve over an algebraically closed
  field, then
\begin{equation*}
\h^r(Y, \bbG_m) =\begin{cases}
\sez (C, \bbG_m) & r=0\\
\pic (Y) & r=1\\
0 & r \equiv 0 \pod{2}\text{, }r\geq 2\\
\bigoplus_{l=1}^N\sfrac{\bbZ}{d_l \bbZ} & r \equiv 1 \pod{2}\text{, }r \geq 3 \text{.}
\end{cases}
\end{equation*}
Moreover, there is a short exact sequence \[0 \rightarrow \pic (C)
\rightarrow \pic (Y) \rightarrow \bigoplus_{l=1}^N\sfrac{\bbZ}{d_l
  \bbZ} \rightarrow 0\text{.}\]
\end{cor}
\begin{proof}
  By the proof of Lemma~\ref{lemma_beta}, we get the description of
  $\h^1(Y, \bbG_m)$.  Moreover, applying Theorem~\ref{theorem_cohom}
  with $G_0=0$, we have $\h^r(Y, \bbG_m)= \bigoplus_{l=1}^N
  H^{r-1}(\sfrac{\bbZ}{d_l \bbZ}, \bbZ)$, for $r \geq 2$, and the
  statement follows from Corollary~\ref{cor_Z}.
\end{proof}
\section{Twisted nodal curves}
Let $Y$ be a twisted nodal curve over an algebraically closed field
$k$ (\cite{AV}, Definition~4.1.2). In particular $Y$ is a connected
tame Deligne-Mumford stack over $k$ with
trivial generic stabilizer. If $C$ is the coarse moduli
space of $Y$, then $C$ is a connected nodal curve over $k$. Let $\pi
\colon Y \rightarrow C$ be the natural morphism; we denote by $S$ the
set of singular points of $C$ and we write $S_Y = S \times_C Y$.
\begin{prop}\label{prop_nodal}
We have $\h^r(C, \bbG_m)=0$ for $r \geq 2$.
\end{prop}
\begin{proof}
Let $\nu \colon \hat{C} \rightarrow C$ be the normalization of $C$,
then, by Theorem~\ref{teorema_milne}, $\h^r(\hat{C}, \bbG_m)=0$ for $r
\geq 2$. Moreover $\nu$ is a finite morphism, therefore $\nu_*$ is an
exact functor (\cite{milne}, II.3.6), and by Leray spectral sequence
for $\nu$, we get $\h^r(C, \nu_* \bbG_{m,\hat{C}}) \cong \h^r(\hat{C},
\bbG_m)$ for $r \geq 0$. There is a natural injective morphism of
sheaves $\bbG_{m,C} \rightarrow \nu_*\bbG_{m, \hat{C}}$, whose
cokernel is concentrated in the singular locus of $C$. Equivalently we
have the following exact sequence \[0 \rightarrow \bbG_{m,C}
\rightarrow \nu_*\bbG_{m, \hat{C}} \rightarrow \bigoplus_{x \in S}x_*
Q_x \rightarrow 0\text{,}\] where $Q_x$ is a sheaf over $x= \spec
k$. Consider the long exact sequence in cohomology \[\cdots
\rightarrow \h^r(C, \bbG_m) \rightarrow \h^r(C, \nu_* \bbG_{m, \hat{C}})
\rightarrow \bigoplus_{x \in S}\h^r(C, x_* Q_x) \rightarrow
\cdots\text{.}\] Since $x \colon \spec k \rightarrow C$ is a closed
embedding, the functor $x_*$ is exact (\cite{milne}, II.3.6) and by
Leray spectral sequence for $x$, we get $\h^r(C, x_* Q_x) \cong
\h^r(\spec k, Q_x)$ for $r \geq 0$. Finally, the groups $\h^r(\spec k,
Q_x)$ vanish for $r \geq 1$, since $k$ is algebraically closed.
\end{proof}
\begin{prop}\label{prop_twisted}
Let $\Sigma$ be the set of closed points of $Y$ with non trivial
stabilizer. Then
\begin{equation*}
\h^r(Y, \bbG_m) =\begin{cases}
\sez (C, \bbG_m) & r=0\\
\pic (Y) & r=1\\
0 & r \equiv 0 \pod{2}\text{, }r\geq 2\\
\bigoplus_{\sigma \in \Sigma}\sfrac{\bbZ}{d_\sigma \bbZ} & r \equiv 1 \pod{2}\text{, }r \geq 3 \text{.}
\end{cases}
\end{equation*}
Moreover, there is a short exact sequence \[0 \rightarrow
\pic (C) \rightarrow \pic (Y) \rightarrow
\bigoplus_{\sigma \in \Sigma}\sfrac{\bbZ}{d_\sigma \bbZ} \rightarrow 0\text{.}\]
\end{prop}
\begin{proof}
  With notations as in the proof of Proposition~\ref{prop_nodal}, let
  $\hat{Y}= Y \times_C \hat{C}$. Notice that the induced morphism
  $\hat{\nu} \colon \hat{Y} \rightarrow Y$ is finite because $\nu$ is
  finite, hence, by Leray spectral sequence for $\hat{\nu}$, we get
  $\h^r(Y, \hat{\nu}_* \bbG_{m, \hat{Y}})= \h^r(\hat{Y}, \bbG_m)$ for
  $r \geq 0$. Moreover $\hat{Y}$ is smooth, hence
  Corollary~\ref{cor_orbicurve} gives a description of the groups
  $\h^r(\hat{Y}, \bbG_m)$. We have the following exact sequence of
  sheaves \[0 \rightarrow \bbG_{m,Y} \rightarrow \hat{\nu}_*\bbG_{m,
    \hat{Y}} \rightarrow Q \rightarrow 0\text{,}\] where $Q=
  \bigoplus_{\sigma \in S_Y}\sigma_*Q_\sigma$ is concentrated in
  $S_Y$. Consider the long exact sequence in cohomology
\begin{equation}\label{eq:a}
\cdots \rightarrow \h^r(Y, \bbG_m) \rightarrow\h^r(Y,
\hat{\nu}_* \bbG_{m, \hat{Y}}) \rightarrow \bigoplus_{\sigma \in
  S_Y}\h^r(Y, \sigma_* Q_\sigma) \rightarrow \cdots\text{.}
\end{equation}
Since $\sigma$ is a closed embedding, by Leray spectral sequence for
$\sigma$ and Proposition~\ref{prop_alg_closed}, we obtain $\h^r(Y,
\sigma_* Q_\sigma) \cong \h^r(\sigma, Q_\sigma)=H^r(G_\sigma,
Q_\sigma^{\sh})$ for $r \geq 0$, where $G_\sigma=
\sfrac{\bbZ}{d_\sigma \bbZ}$ is the stabilizer
of $\sigma$ and $Q_\sigma^{\sh}$ is the strictly henselian local ring
of $Q$ at $\sigma$. Notice that $Q_\sigma^{sh}=
\sfrac{\hat{A}^*}{A^*}$, where $A = \scrO_{Y, \sigma}^{\sh}$ and
$\spec \hat{A}$ is the normalization of $\spec A$. Since $Y$ is
\'etale locally a nodal curve over $k$, we have that $\hat{A}^*=
k^*\times k^*$ and $A^*=k^*$, therefore $Q_\sigma^{\sh}=k^*$. By
Theorem~\ref{prop_cohom_groups}, we have
\begin{equation*}
H^r(G_\sigma, Q_\sigma^{\sh})=
\begin{cases}
k^* & r=0\\
\sfrac{\bbZ}{d_\sigma \bbZ} & r \equiv 1 \pod{2}\\
0 & r \equiv 0 \pod{2}\text{, }r\geq 2\text{.}
\end{cases}
\end{equation*}
We substitute these results in~\eqref{eq:a} and obtain, for $r \geq
1$, \[0 \rightarrow \h^{2r+1}(Y, \bbG_m) \rightarrow \bigoplus_{\sigma
  \in \hat{\Sigma}}\sfrac{\bbZ}{d_\sigma \bbZ} \xrightarrow{\rho}
\bigoplus_{\sigma \in S_Y}\sfrac{\bbZ}{d_\sigma \bbZ} \rightarrow
\h^{2r+2}(Y, \bbG_m) \rightarrow 0\text{,}\] where $\hat{\Sigma}$ is
the set of closed points of $\hat{Y}$ with non trivial stabilizer and
the map $\rho$ is described as follows. If $\sigma \in S_Y$, with
$d_\sigma >0$, then $\sigma \times_Y \hat{Y}$ consists of two points
$\sigma_1, \sigma_2 \in \hat{\Sigma}$; let $\rho_\sigma$ be the
restriction of $\rho$ to $\sfrac{\bbZ}{d_{\sigma_1} \bbZ}\oplus
\sfrac{\bbZ}{d_{\sigma_2}\bbZ}$, then $\rho_\sigma(a,b)= a-b \in
\sfrac{\bbZ}{\sigma \bbZ}$. Otherwise, the restriction of $\rho$ to
$\sfrac{\bbZ}{d_\sigma \bbZ}$ is zero. It follows that, for every $r
\geq 3$,
\begin{equation*}
\h^r(Y, \bbG_m)=
\begin{cases}
\bigoplus_{\sigma \in \Sigma}\sfrac{\bbZ}{d_\sigma \bbZ} &r \equiv 1 \pod{2}\\
0 & r \equiv 0 \pod{2}\text{.}
\end{cases}
\end{equation*}
Moreover we have the following commutative diagram with exact rows
  \[
  \begin{tikzpicture}[xscale=2.0,yscale=-1.6]
    \node (A0_0) at (0, 0) {$0$};
    \node (A0_1) at (0.8, 0) {$\sez (C, \bbG_m)$};
    \node (A0_2) at (1.9, 0) {$\sez (\hat{C}, \bbG_m)$};
    \node (A0_3) at (3, 0) {$\displaystyle{\bigoplus_{\sigma \in S_Y} k^*}$};
    \node (A0_4) at (3.9, 0) {$\pic (C)$};
    \node (A0_5) at (4.8, 0) {$\pic (\hat{C})$};
    \node (A0_6) at (5.8, 0) {$0$};
    \node (A1_0) at (0, 1) {$0$};
    \node (A1_1) at (0.8, 1) {$\h^0(Y, \bbG_m)$};
    \node (A1_2) at (1.9, 1) {$\sez (\hat{C}, \bbG_m)$};
    \node (A1_3) at (3, 1) {$\displaystyle{\bigoplus_{\sigma \in S_Y} k^*}$};
    \node (A1_4) at (3.9, 1) {$\pic (Y)$};
    \node (A1_5) at (4.8, 1) {$\pic (\hat{Y})$};
    \node (A1_6) at (5.8, 1) {$\displaystyle{\bigoplus_{\sigma \in S_Y}\sfrac{\bbZ}{d_\sigma \bbZ}}$};
    \node (A1_7) at (7, 1) {$\h^2(Y, \bbG_m)$};
    \node (A1_8) at (7.8, 1) {$0$};
    \path (A1_4) edge [->]node [auto] {$\scriptstyle{}$} (A1_5);
    \path (A0_1) edge [->]node [auto] {$\scriptstyle{}$} (A1_1);
    \path (A1_7) edge [->]node [auto] {$\scriptstyle{}$} (A1_8);
    \path (A0_0) edge [->]node [auto] {$\scriptstyle{}$} (A0_1);
    \path (A0_1) edge [->]node [auto] {$\scriptstyle{}$} (A0_2);
    \path (A1_0) edge [->]node [auto] {$\scriptstyle{}$} (A1_1);
    \path (A0_3) edge [-,double distance=1.5pt]node [auto] {$\scriptstyle{}$} (A1_3);
    \path (A0_6) edge [->]node [auto] {$\scriptstyle{}$} (A1_6);
    \path (A1_1) edge [->]node [auto] {$\scriptstyle{}$} (A1_2);
    \path (A1_5) edge [->]node [auto] {$\scriptstyle{}$} (A1_6);
    \path (A0_4) edge [->]node [auto] {$\scriptstyle{}$} (A0_5);
    \path (A0_3) edge [->]node [auto] {$\scriptstyle{}$} (A0_4);
    \path (A0_4) edge [->]node [auto] {$\scriptstyle{}$} (A1_4);
    \path (A0_5) edge [->]node [auto] {$\scriptstyle{}$} (A0_6);
    \path (A0_2) edge [-,double distance=1.5pt]node [auto] {$\scriptstyle{}$} (A1_2);
    \path (A1_6) edge [->]node [auto] {$\scriptstyle{}$} (A1_7);
    \path (A1_2) edge [->]node [auto] {$\scriptstyle{\rho_0}$} (A1_3);
    \path (A0_2) edge [->]node [auto] {$\scriptstyle{\rho_0}$} (A0_3);
    \path (A0_5) edge [->]node [auto] {$\scriptstyle{}$} (A1_5);
    \path (A1_3) edge [->]node [auto] {$\scriptstyle{}$} (A1_4);
  \end{tikzpicture}
  \]
which implies $\h^0(Y, \bbG_m)= \sez(C, \bbG_m)$ and $\h^2(Y, \bbG_m)=0$. Finally, by snake lemma, 
the following sequence is exact \[0 \rightarrow \pic (C)
\rightarrow \h^1(Y, \bbG_m) \rightarrow \bigoplus_{\sigma \in \Sigma}
\sfrac{\bbZ}{d_\sigma \bbZ}\rightarrow 0\text{.}\qedhere\]
\end{proof}
\subsection{Banded gerbes over orbicurves}\label{subsection_abeliangerbes}
Recall that, if $\scrG$ is a sheaf of abelian groups on a twisted
nodal curve $Y$, then $\h^2(Y, \scrG)$ classifies the $\scrG$-gerbes
over $Y$ banded by $\scrG$ (\cite{giraud}, IV.3.5).  Let $G_0$ be an
abelian finite group of order not divided by $\Char k$ and let $G_0 =
\bigoplus_{h=1}^M \sfrac{\bbZ}{d_h \bbZ}$ be a decomposition of $G_0$
as a direct sum of cyclic groups. We have a short exact sequence \[0
\rightarrow G_0 \rightarrow \bigoplus_{h=1}^M\bbG_m \rightarrow
\bigoplus_{h=1}^M\bbG_m \rightarrow 0\text{,}\] deduced by Kummer
sequence for each factor. From the induced long exact sequence in
cohomology we get \[\bigoplus_{h=1}^M \pic (Y) \xrightarrow{\psi}
\h^2(Y, G_0)= \bigoplus_{h=1}^M \h^2(Y, \sfrac{\bbZ}{d_h \bbZ})
\rightarrow \bigoplus_{h=1}^M \h^2(Y, \bbG_m)=0\text{,}\] where,
according to \cite{cadman} (Section~2.4), the map $\psi$ associates to
$(L_1, \ldots, L_M) \in \bigoplus_{l=1}^M \pic (Y)$ the $G_0$-gerbe
$\sqrt[\leftroot{-1}\uproot{3}d_1]{\sfrac{L_1}{Y}}\times_Y \cdots
\times_Y \sqrt[\leftroot{-1}\uproot{3}d_M]{\sfrac{L_M}{Y}}$. Therefore
every $G_0$-gerbe over $Y$ banded by $G_0$ is obtained as a finite
number of root constructions.
\section{Trivial gerbes}
Throughout this section we will use the notations of
Section~\ref{section_setting}.
\begin{prop}\label{prop_trivial_gerbe}
We have $\h^1(Y \times \B G_0, \bbG_m) = \pic (Y)
\oplus \Hom (G_0, K^*)$ and, for $r\geq 2$, there are short exact sequences \[0 \rightarrow Q^r \rightarrow \h^r(Y \times
\B G_0, \bbG_m) \rightarrow \bigoplus_{l=1}^N \sfrac{H^{r-1}(G_0
  \times \sfrac{\bbZ}{d_l \bbZ}, \bbZ)}{H^{r-1}(G_0, \bbZ)}
\rightarrow 0\text{,}\] where $Q^r$ fits in the
following exact sequence \[0 \rightarrow H^r(G_0,\sez (C, \bbG_m))
\rightarrow Q^r \rightarrow H^{r-1}(G_0, \pic(Y)) \rightarrow
0\text{.}\]
\end{prop}
\begin{proof}
The map $Y \xrightarrow{\rho} Y \times \B G_0$ obtained by
$\spec k \rightarrow \B G_0$ is a Galois cover with group $G_0$. Then
we can consider the Hochschild-Serre spectral sequence
\begin{equation}\label{eq:succ_hs_Y}
E_2^{p,q}=H^p(G_0, \h^q(Y, \bbG_m)) \Rightarrow \h^{p+q}(Y \times \B
G_0, \bbG_m)\text{.}
\end{equation}
By Corollary~\ref{cor_orbicurve},
\begin{equation*}
E_2^{p,q}=
\begin{cases}
H^p(G_0, \sez{C, \bbG_m}) & q=0\\
H^p(G_0, \pic (Y)) & q=1\\
\bigoplus_{l=1}^N H^p(G_0, H^{q-1}(\sfrac{\bbZ}{d_l\bbZ},\bbZ)) & q \geq 2\text{.} 
\end{cases}
\end{equation*}
In particular, the groups $E_2^{p,q}$ and the maps between them
coincide, for $q \geq 2$, with the groups $F_2^{p,q-1}$ and relative
maps for the Hochschild-Serre spectral sequence
\begin{equation}\label{eq:succ_hs_G}
F_2^{p,q}= \bigoplus_{l=1}^N H^p(G_0, H^q(\sfrac{\bbZ}{d_l \bbZ},
\bbZ)) \Rightarrow \bigoplus_{l=1}^N H^{p+q}(G_0 \times
\sfrac{\bbZ}{d_l \bbZ}, \bbZ)\text{.}
\end{equation}
By Lemma~\ref{lemma_zero_maps}, we have $E_\infty^{p,q}=
E_2^{p,q}$. Comparing filtrations for
\eqref{eq:succ_hs_Y} and
\eqref{eq:succ_hs_G}, we get \[0
\rightarrow H_{r-1}^r \rightarrow \h^r(Y \times \B G_0, \bbG_m)
\rightarrow \bigoplus_{l=1}^N \sfrac{H^{r-1}(\sfrac{\bbZ}{d_l
    \bbZ}\times G_0,\bbZ)}{H^{r-1}(G_0,\bbZ)} \rightarrow 0\text{,}\]
for $r \geq 1$, where $H_{r-1}^r$ fits in the following exact sequence \[0 \rightarrow
H^r(G_0, \sez (C, \bbG_m)) \rightarrow H_{r-1}^r \rightarrow
H^{r-1}(G_0, \pic(Y)) \rightarrow 0\text{.}\]
Consider the natural morphism $Y \times \B G_0
\xrightarrow{\pi} Y$. We have the following commutative diagram
  \[
  \begin{tikzpicture}
    \def\x{3.0}
    \def\y{-1.4}
    \node (A0_0) at (0.5*\x, 0*\y) {$0$};
    \node (A0_1) at (1*\x, 0*\y) {$\pic (Y)$};
    \node (A0_2) at (2*\x, 0*\y) {$\h^1(Y \times \B G_0, \bbG_m)$};
    \node (A0_3) at (3.1*\x, 0*\y) {$\Hom (G_0, K^*)$};
    \node (A0_4) at (3.8*\x, 0*\y) {$0$};
    \node (A1_0) at (0.5*\x, 1*\y) {$0$};
    \node (A1_1) at (1*\x, 1*\y) {$\pic (Y)$};
    \node (A1_2) at (2*\x, 1*\y) {$\h^1(Y \times \B G_0, \bbG_m)$};
    \node (A1_3) at (3.1*\x, 1*\y) {$\Hom (G_0, K^*)$};
    \node (A1_4) at (3.8*\x, 1*\y) {$0$};
    \path (A0_3) edge [->] node [auto] {$\scriptstyle{}$} (A0_2);
    \path (A1_3) edge [-, double] node [auto] {$\scriptstyle{}$} (A0_3);
    \path (A1_1) edge [->] node [auto] {$\scriptstyle{\pi^*}$} (A1_2);
    \path (A0_2) edge [->] node [above] {$\scriptstyle{\rho^*}$} (A0_1);
    \path (A0_4) edge [->] node [auto] {$\scriptstyle{}$} (A0_3);
    \path (A1_0) edge [->] node [auto] {$\scriptstyle{}$} (A1_1);
    \path (A0_2) edge [-,double] node [auto] {$\scriptstyle{}$} (A1_2);
    \path (A0_1) edge [->] node [auto] {$\scriptstyle{}$} (A0_0);
    \path (A0_1) edge [-,double] node [left] {$\scriptstyle{{(\pi \circ \rho)}_*}$} (A1_1);
    \path (A1_2) edge [->] node [auto] {$\scriptstyle{}$} (A1_3);
    \path (A1_3) edge [->] node [auto] {$\scriptstyle{}$} (A1_4);
  \end{tikzpicture}
  \]
where the second row is given by Theorem~\ref{theorem_cohom} and the
first row is deduced from the argument above, after
noticing that $\Hom (G_0, \sez (C, \bbG_m)) = \Hom (G_0,
K^*)$. Hence we get \[\h^1(Y \times \B G_0, \bbG_m) = \pic (Y)
\oplus \Hom (G_0, K^*)\text{.}\qedhere\]
\end{proof}
\begin{oss}\label{remark_cyclic_gerbe}
If $G_0$ is cyclic then, by Theorem~\ref{prop_cohom_groups} and
Proposition~\ref{prop_trivial_gerbe}, we have, for $r\geq 1$, \[\h^{2r+1}(C \times \B G_0,
\bbG_m)= \Hom (G_0, \sez(C, \bbG_m)) \oplus H^2(G_0, \pic
(C))\text{.}\]
\end{oss}
\begin{cor}\label{cor_trivial_gerbe}
If $C$ is projective and $G_0$ is cyclic, then
\begin{equation*}
\h^r(C \times \B G_0, \bbG_m)=
\begin{cases}
k^* & r=0\\
\pic (C)\oplus G_0 & r=1\\
{(G_0)}^{2g} & r \equiv 0 \pod{2} \text{, } r\geq 2\\
G_0 \oplus G_0& r \equiv 1 \pod{2}\text{, }r \geq 3\text{.}
\end{cases}
\end{equation*}
\end{cor}
\begin{proof}
  Let $d$ be the order of $G_0$.  Recall that there is a short exact
  sequence \[0 \rightarrow \pic^0(C) \rightarrow \pic(C) \rightarrow
  \bbZ \rightarrow 0\text{.}\] Moreover, since $C$ is projective,
  $\sez (C, \bbG_m) =k^*$ and the sequence \[0 \rightarrow
  {(\sfrac{\bbZ}{d \bbZ})}^{2g}\rightarrow \pic^0(C) \xrightarrow{d}
  \pic^0(C) \rightarrow 0\] is exact, where the map $d$ is defined by
  $d(a) = a^{d}$, for all $a \in \pic^0(C)$ (see \cite{mumford},
  Section~IV.21, Lang's Theorem). Therefore, applying
  Theorem~\ref{prop_cohom_groups}, we obtain, for $r\geq 1$
\begin{equation*}
\begin{cases}
H^{2r-1}(G_0,k^*)=G_0\text{, }&H^{2r}(G_0, k^*)=0\text{,}\\
H^{2r-1}(G_0, \pic^0(C)) = {(G_0)}^{2g}\text{, }&H^{2r}(G_0, \pic^0(C)) = 0\text{,}\\
H^{2r-1}(G_0, \pic(C))={(G_0)}^{2g}\text{, }&H^{2r}(G_0, \pic(C))=G_0\text{.}\\
\end{cases}
\end{equation*}
Then the statement follows from Proposition~\ref{prop_trivial_gerbe}
and Remark~\ref{remark_cyclic_gerbe}.
\end{proof}
\section{Examples}\label{section_example}
In this section we present two examples: the first one shows
that, in general, the groups $\h^r(X, \bbG_m)$ may depend
on the gerbe $X \rightarrow Y$, for $r \geq 3$; the second one shows
that $\h^2(X, \bbG_m)$ may depend on the gerbe $X \rightarrow Y$ if the
stabilizers are not abelian.

Let $p$ be a prime integer and consider $\bbA_\bbC^1$ with the action
$\bbC[x]\times \sfrac{\bbZ}{p \bbZ} \rightarrow \bbC [x]$ given by
$(x, \lambda) \mapsto \lambda x$, for all $\lambda \in \sfrac{\bbZ}{p
  \bbZ}$. Then $Y=[\sfrac{\bbA_\bbC^1}{\mu_p}]$ is an orbicurve.

\paragraph{}
Let $X=[\sfrac{\bbA_\bbC^1}{\mu_{p^2}}]$, where the action is given by
$(x, \lambda)\mapsto\lambda^p x$, for $\lambda \in \sfrac{\bbZ}{p^2
  \bbZ}$. By Kummer sequence, $\h^2(Y, \mu_p)=\sfrac{\bbZ}{p \bbZ}$
and, since $p$ is prime, there are only two non isomorphic banded
$\mu_p$-gerbes over $Y$ (\cite{giraud}, IV.3.5); in particular, by
Theorem~\ref{theorem_AOV}, $X$ is the only non trivial $\mu_p$-gerbe
banded over $Y$ (up to isomorphism). By Corollary~\ref{cor_cyclic},
$\h^{2r+1}(X, \bbG_m)$ is finite of order $p^3$ whereas, by
Proposition~\ref{prop_trivial_gerbe}, $\h^{2r+1}(Y \times \B \mu_p,
\bbG_m)$ has order $p^2$, for $r \geq 1$.

\paragraph{}
Set $p=2$. Let $D_{2m}$ be the dihedral group with $2m$ elements, with
$m$ odd, $m \geq 3$. We denote by $r \in \sfrac{\bbZ}{m \bbZ}$ and $s
\in \sfrac{\bbZ}{2 \bbZ}$ the generators of $D_{2m}$. Let $X=
[\sfrac{\bbA_\bbC^1}{D_{2m}}]$, where the action is given by $(x,
r^is^\epsilon) \mapsto {(-1)}^\epsilon x$, for $1 \leq i \leq m$ and
$\epsilon = 0,1$. By \cite{weibel}, 6.7.10, we have $H^2(D_{2m},
\bbZ)=\sfrac{\bbZ}{2 \bbZ}$, hence $\tau^{(2)}\colon \sfrac{\bbZ}{m
  \bbZ} \rightarrow \sfrac{\bbZ}{2 \bbZ}$ is the zero map. Therefore,
by Remark~\ref{remark_beta} and Proposition~\ref{prop_trivial_gerbe},
$\h^2(X, \bbG_m)= \ker \beta^{(2)}=\ker \tau^{(2)}= \sfrac{\bbZ}{m
  \bbZ}$ whereas $\h^2(Y \times \B \mu_m, \bbG_m)=0$.

\end{document}